\documentclass[11pt]{amsart}
\usepackage{amsmath}
\usepackage{amsfonts}

\numberwithin{equation}{section}

\newtheorem{Theorem}{Theorem}

\numberwithin{Theorem}{section}

\newtheorem{Lemma}[Theorem]{Lemma}

\newtheorem{Proposition}[Theorem]{Proposition}

\newtheorem{Corollary}[Theorem]{Corollary}

\title[Regularity of the free boundary]{Regularity
of the free boundary in a nonlocal one-dimensional
parabolic free boundary value problem}

\author{Rossitza Semerdjieva}

\address{%
Institute of Mathematics and Informatics\\
Bulgarian Academy of Sciences\\
Acad. G. Bonchev Str., Bl. 8\\
1113 Sofia, Bulgaria}

\email{rsemerdjieva@yahoo.com, rossitza@math.bas.bg}

\begin{document}

\begin{abstract}
We consider one-dimensional parabolic free boundary value problem
with a nonlocal (integro-differential) condition on the free
boundary. Results on $C^m$-regularity  of the free boundary are
obtained. In particular, a necessary and sufficient condition for
infinite differentiability of the free boundary is given.
\end{abstract}



 \maketitle

\section{Introduction}

In this paper we study the regularity properties of the free boundary
in the following one-dimensional parabolic free boundary value
problem.

{\em Problem P}.  Find $s(t)>0$ and $u(x,t) $ such that
\begin{equation}
\label{1}  u_t = u_{xx}- \lambda \,u, \quad  \lambda =const >0, \quad
0 < x < s(t), \;\; t>0,
\end{equation}
\begin{equation}
\label{3} u (0, t) = f(t),   \quad  t\geq 0,
\end{equation}
\begin{equation}
\label{4}  u(x,0) = \varphi (x),   \;\; x\in [0,b], \;\; s(0)=b >0,
\quad  \varphi (0) = f(0),
\end{equation}
\begin{equation}
\label{5} u_x (s(t), t) = 0, \quad t>0,
\end{equation}
\begin{equation}
\label{2}  s^\prime (t) =  \int_0^{s(t)}  (u(x,t)-\sigma)
  dx,   \quad \sigma = const >0, \;  \;  t > 0.
\end{equation}

Notice that (\ref{3})--(\ref{5}) are mixed type boundary conditions
for the parabolic equation (\ref{1}), and  (\ref{2}) is an
integro-differential condition on the free boundary $x= s(t). $
Similar free boundary value problems arise in tumor modeling and
modeling of nanophased thin films (see \cite{CF02, FR99, Fr07, CGP}).

Our goal in this paper is to examine the relationship between the
smoothness of the functions $f(t) $ and $s(t),$ and to show the
essential impact of the nonlocal character of condition (\ref{2}) on
the regularity properties of the free boundary.

If one sets $\lambda=0$ in (\ref{1}), and replaces our conditions
(\ref{5}) and (\ref{2}) by the conditions
\begin{equation}
\label{5a} u (s(t), t) = 0 \quad \text{for} \;\; t>0
\end{equation}
and
\begin{equation}
\label{2a} s^\prime (t) =- u_x (s(t),t)\quad \text{for} \;\; t > 0
\end{equation}
respectively, the resulting problem (\ref{1}) (with $\lambda=0$),
(\ref{3}), (\ref{4}), (\ref{5a}) and (\ref{2a}) is the classical
one-dimensional Stefan problem (see \cite[Ch. 8]{F}, \cite{Rub},
\cite[Ch. 17]{C84}). In this context, the infinite differentiability
of the free boundary has been established in \cite{CH68, CP71, CLS,
Sch}.

On the other hand, in \cite{F76} it is proved that if $f(t)$ is an
analytic function then $s(t), \; t>0$ is also an analytic function.
In addition, the analyticity of $s(t)$ at $t=0 $ is studied in
\cite{FPR80}.

There is a vast literature on the regularity of free boundaries in
multi-dimensional (multi-phase) Stefan problems and their
generalizations (e.g., see \cite{KindNir, EsSim, PSS} and the
bibliography therein). But in general the one-dimensional free
boundary problems cannot be treated as a partial case of
multidimensional ones, and their handling requires specific methods.

Following the method in \cite{Sch}, our approach in studying the
regularity of the free boundary in Problem~$P$ is based on the
theory of anisotropic H\"older spaces. In Section 2 we estimate from
below the H\"older and $C^m$-smoothness of the free boundary. In
Theorem~\ref{thm1}, we prove that if $f(t)$ has continuous
derivatives up to order $m$ on $(0,T], \; T> 0, $ then $s(t)$ has
continuous derivatives up to order $m+1$ on $(0,T].$ Therefore, if
$f(t) $ is infinitely differentiable on $(0,\infty)$, it follows
that $s(t)$ is infinitely differentiable on $(0,\infty)$ as well.

However, it turns out that $s(t) $ may not have derivatives of order
higher than two if we assume $f(t) \in C^1 ([0,\infty))$ only (see
Section~3, where we estimate the smoothness of $s(t)$ from above).
More generally, in Theorem~\ref{thm2} we prove that if $s(t)$ has on
$(0,T]$ continuous derivatives up to order $m+2$ then $f(t)$ has
continuous derivatives up to order $m$ on $(0,T].$ Therefore, if
$f(t) $ is not infinitely differentiable on $(0, \infty),$ then the
free boundary is not infinitely differentiable curve as well.

This is in a striking contrast with the case of one-dimensional
Stefan problem, where the infinite differentiability of the free
boundary does not require infinite differentiability of the boundary
data at $x=0$ (see \cite{CH68}, \cite{CP71}, \cite{CLS}, \cite{Sch}).
In our Problem $P,$ due to the nonlocal character of condition
(\ref{2}), the smoothness of the free boundary is essentially related
to the smoothness of $f(t),$ namely the free boundary is an
infinitely differentiable curve if and only if the function $f(t)$ is
infinitely differentiable.

\section{Lower bounds for the smoothness of the free boundary}
Results on global existence and uniqueness of classical solutions of
Problem $P$ are obtained in \cite[Theorem 1.1]{R3} (see also
\cite{R1}, \cite{R2}). More precisely, the following holds.

{\em Global solvability of Problem $P:$ Suppose
\begin{eqnarray}
\label{6} f(t) \in C^1([0,\infty)), \quad \varphi (x) \in C^2([0,b]), \quad f(0) =\varphi (0),\\
\nonumber   f^\prime (0) = \varphi^{\prime \prime} (0)-\lambda
\varphi (0), \qquad \varphi^\prime (b) =0.
\end{eqnarray}
Then there exists a unique pair of functions $u (x,t)$ and $s(t)
$ such that \\
(i) \hspace{1mm} $u (x,t) $ is defined, continuous and has continuous
partial derivatives $u_x, \, u_t, \, u_{xx} $
 in the domain
$\{(x,t): \;  0 \leq x \leq s(t), \, t \geq 0 \};$\\
(ii) \hspace{1mm}  $s(t) \in C^1 ([0,\infty)),
\quad s(t) > 0 \; \; \text{for} \;t\geq 0; $\\
(iii) the conditions (\ref{1})--(\ref{2}) hold.} \vspace{2mm}

Let the pair of functions $(u(x,t), s(t))$ be a classical solution of
Problem~$P$ satisfying (i)--(iii). It is easy to see that $s(t)\in
C^2 ([0, \infty)).$ Indeed, since $u_t (x,t)$ is defined and
continuous for $ 0 \leq x \leq s(t), \, t>0, $ from (\ref{2}) it
follows
$$
s^{\prime \prime}(t) = (u(s(t),t)- \sigma) s^\prime (t)+
\int_0^{s(t)} u_t (x,t) dx.
$$
By (\ref{1}) and (\ref{2}),
\begin{align*}
\int_0^{s(t)} u_t (x,t) dx&= \int_0^{s(t)}  (u_{xx} (x,t)-\lambda
u(x,t)) dx\\ &= u_x (s(t),t) -u_x (0,t)-\lambda s^\prime (t) -\lambda
\sigma s(t),
\end{align*}
so using (\ref{5}) we obtain
\begin{equation}
\label{10} s^{\prime \prime}(t) = (u(s(t),t)- \lambda -\sigma) \,
s^\prime (t) -\lambda \sigma s(t) - u_x (0,t),
\end{equation}
where the expression on the right is a continuous function for $t\geq
0,$ i.e., $s (t) \in C^2 ([0,\infty)).$

In this section, our main result is the following statement.
\begin{Theorem}
\label{thm1} Suppose the pair of functions $(u(x,t), s(t))$ is a
classical solution of Problem $P$ satisfying (i)--(iii).
 If $f(t) \in C^m ((0,T]),$ where  $m \in \mathbb{N}, \, m\geq 2$
and $T=const >0, $ then $s(t) \in C^{m+1} ((0,T]).$

In particular, if $f(t) \in C^\infty ((0,\infty)),$ then $s(t) \in
C^\infty ((0,\infty)).$
\end{Theorem}

In the proof of Theorem~\ref{thm1} we need some preliminary results.
First we use the change of variables $ \xi=\frac{x}{s(t)}, \, t=t  $
to transform (\ref{1}) to an equation in a cylindrical domain by
setting
\begin{equation}
\label{11.0}  v(\xi,t)=u(\xi s(t),t),\quad Q=\{(\xi,t):\,0<\xi<1,\;
t>0\}.
\end{equation}
Then, in view of (i)--(iii), it follows that $v,\, v_\xi,\, v_t,\,
v_{\xi\xi} \in C(\overline{Q})$ and
 \begin{equation}
\label{11.1}
 Lv:=v_t-\frac{1}{s^2}v_{\xi
 \xi} - \frac{{\xi}s^{\prime}}{s}v_\xi - \lambda v=0 \quad \text{for} \; \; (\xi,t)\in
 Q ,
 \end{equation}
\begin{equation}
\label{11.2} v(0,t)=f(t),\quad v(\xi,0)=\varphi({\xi}s(0)),\quad
 v_{\xi}(1,t)=0.
 \end{equation}
 From (\ref{10}) we obtain
 \begin{equation}
\label{11.3}
 s^{\prime\prime}(t)=(v(1,t)- \lambda-\sigma) \, s^{\prime}(t)-\lambda
\sigma s(t) - \frac{1}{s(t)}v_{\xi}(0,t),\quad t\geq 0.
 \end{equation}
For convenience, we set
 \begin{equation}
\label{0.0}  Q^{\varepsilon, T}_{\delta_1, \delta_2} =\{(\xi,t) : \;
\delta_1 <\xi <\delta_2, \;  \varepsilon < t < T   \}.
\end{equation}

In order to prove Theorem~\ref{thm1} we are going to estimate the
H\"older smoothness of $v(\xi, t) $ and $s(t) $ in terms of the
H\"older smoothness of $f(t).$  To this end we use anisotropic
H\"older spaces $H^{m+\ell, \frac{m+\ell}{2}} \left
(\overline{Q^{\varepsilon, T}_{\delta_1, \delta_2}} \right ),$ where
$m=0,1,2, \ldots$ and $ \ell \in (0,1).$

Recall that $H^{m+\ell, \frac{m+\ell}{2}} \left (\overline{G} \right
)$  is the Banach space of all functions $v(\xi,t)$  that are
continuous on $\overline{G}$ together with all derivatives of the
form $D^k_\xi D^r_t v$ for $k+2r\leq m $ and have a finite norm
$$
\|v\|_G^{(m+l)}= \sum_{j=0}^m \sum_{k+2r=j} \left | D^k_\xi D^r_t v
\right |^{(0)} +\sum_{k+2r=m} \langle D^k_\xi D^r_t v
\rangle_{\xi,G}^{(\ell)} $$
$$
+\sum_{k+2r=m} \langle D^k_\xi D^r_t v \rangle_{t,G}^{(\ell/2)} +
\sum_{k+2r=m-1} \langle D^k_\xi D^r_t v \rangle_{t,G}^{\left (
\frac{1+\ell}{2}\right ) },
$$
where $G$ is a bounded rectangular domain, $\langle v
\rangle_{\xi,G}^{\ell} $ and $\langle v\rangle_{t,G}^{\ell} $ are the
H\"older constants of a function $v(\xi,t) $ in $\xi$ and $t$
respectively in the domain $\overline{G}$ with the exponent $\ell, \;
\ell \in (0,1), $ and $
 | D^k_\xi D^r_t v  |^{(0)} = \max_{\overline{G}} |
D^k_\xi D^r_t v |. $ For more details about these definitions and
notations we refer to the book \cite[Intr., p. 7]{LUS}.

In the following the functions of one variable $t$ are regarded as
functions of two variables $x$ and $t.$
\begin{Proposition}
\label{prop1}

(a) For every $T>0, \, \ell \in (0,1)$ we have
 \begin{equation}
\label{p.0} v(\xi,t) \in H^{1+\ell,\frac{1+\ell}{2}}
 \left (\overline{Q^{0, T}_{0,1}} \right ).
 \end{equation}

(b) If $f(t) \in H^{m+\ell,\frac{m+\ell}{2}}
 \left (\overline{Q^{\varepsilon, T}_{0,1}} \right ), $
 where $m\in \mathbb{N}, \, m\geq 2, \; \ell \in (0,1) $ and $
T> \varepsilon >0, $    then
\begin{equation}
\label{p.1} v(\xi,t) \in H^{m+\ell,\frac{m+\ell}{2}}
 \left (\overline{Q^{\tilde{\varepsilon}, T}_{0,1}} \right ) \quad
 \forall \tilde{\varepsilon} \in (\varepsilon, T),
  \end{equation}
  and
\begin{equation}
\label{p.2}  s(t) \in H^{m+3+\ell,\frac{m+3+\ell}{2}}\left
(\overline{Q^{\tilde{\varepsilon}, T}_{0,1}} \right ) \quad
 \forall \tilde{\varepsilon} \in (\varepsilon, T).
 \end{equation}
\end{Proposition}

 The following lemma helps to make  the inductive step in the proof of
Proposition~\ref{prop1}.
\begin{Lemma}
\label{lem1} Let $\varepsilon>0, $  $\ell \in (0,1), \; k \in
\mathbb{N}, $
\begin{equation}
\label{21.0} w(\xi, t) \in H^{k+\ell,\frac{k+\ell}{2}}
 \left (\overline{Q^{\varepsilon, T}_{0,1}} \right )
 \end{equation}
and $w_{\xi \xi} $ exists in $ Q^{\varepsilon, T}_{0,1}$ in the case
$k=1,$ and let $w(\xi,t)$ satisfy the equation
 \begin{equation}
\label{21.1} \tilde{L} w:=w_t - a(\xi,t) w_{\xi \xi} - b (\xi, t)
w_\xi - c (\xi, t) w = F(\xi,t), \quad  (\xi, t) \in Q^{\varepsilon,
T}_{0,1},
\end{equation}
where
 \begin{equation}
\label{21.3} a, b, c, F  \in H^{k-1+\ell,\frac{k-1+\ell}{2}}
 \left (\overline{Q^{\varepsilon, T}_{0,1}} \right ), \quad a(\xi,t) \geq const
 >0.
  \end{equation}
If
 \begin{equation}
\label{21.2} w(0,t)= f(t), \quad w_\xi (1,t) = g(t), \quad
\varepsilon \leq t \leq T,
\end{equation}
with
 \begin{equation}
\label{21.4} f(t) \in H^{k+1+\ell,\frac{k+1+\ell}{2}}
 \left (\overline{Q^{\varepsilon, T}_{0,1}} \right ), \quad
 g(t)  \in H^{k+\ell,\frac{k+\ell}{2}}
 \left (\overline{Q^{\varepsilon, T}_{0,1}} \right ),
 \end{equation}
then
\begin{equation}
\label{21.5} w(\xi, t) \in H^{k+1+\ell,\frac{k+1+\ell}{2}}
 \left (\overline{Q^{\tilde{\varepsilon}, T}_{0,1}} \right ) \quad \forall
 \tilde{\varepsilon} \in (\varepsilon, T).
 \end{equation}
\end{Lemma}

 \begin{proof}  Fix  $\; \tilde{\varepsilon} \in (\varepsilon,T) $
and choose arbitrary  $\varepsilon_1 \in
(\varepsilon,\tilde\varepsilon) $ and $\delta \in (1/2, 1). $
  Consider a function $\psi(\xi,t)$ of the form
$ \psi(\xi,t)=\psi_{1}(\xi)\,\psi_{2}(t), $  where $ \psi_1, \psi_2
\in C^{\infty}(\mathbb{R}), \; $  $  0 \leq \psi_1 (\xi), \psi_2 (t)
\leq 1 $ and
\begin{equation}
\label{psi}
 \psi_1 (\xi) = \begin{cases}   1   &  \text{if} \;\;
\xi \leq 1/2 \\   0  & \text{if}  \; \;  \xi \geq \delta
\end{cases},   \qquad
\psi_2 (t) = \begin{cases}   1   &  \text{if}  \; \;   t \geq \tilde{\varepsilon} \\
0  & \text{if}  \; \;  t \leq \varepsilon_1
\end{cases}.
\end{equation}
Then the function $w_1 (\xi, t) = w (\xi, t) \, \psi (\xi, t)$ is a
solution of the boundary value problem
 \begin{equation}
\label{31.1} \tilde{L} w_1 = F_1(\xi,t), \quad (\xi, t) \in
Q^{\varepsilon, T}_{0,\delta},
\end{equation}
 \begin{equation}
\label{31.2} w_1 (0,t) = f(t) \psi_2 (t), \quad w_1 (\delta,t) = 0,
\;\; \varepsilon  \leq t \leq T,
\end{equation}
 \begin{equation}
\label{31.3} w_1 (\xi, \varepsilon) = 0, \quad 0 \leq \xi \leq
\delta,
\end{equation}
where
\begin{equation}
\label{31.4} F_1 (\xi, t) = F\, \psi + w \, \left ( \psi_t - a
\psi_{\xi \xi} - b \psi_\xi \right )  - 2a w_\xi \psi_\xi .
\end{equation}

In view of (\ref{21.0}) and (\ref{21.3}), it follows that $F_1(\xi,
t) \in H^{k-1+\ell,\frac{k-1+\ell}{2}}  \left
(\overline{Q^{\varepsilon, T}_{0,\delta}} \right ).$ Therefore,
applying Theorem 5.2 of \cite[Ch.4]{LUS}, we conclude that
Problem~(\ref{31.1})--(\ref{31.3}) has a unique solution $$w_1
(\xi,t) \in H^{k+1+\ell,\frac{k+1+\ell}{2}}
 \left (\overline{Q^{\varepsilon, T}_{0,\delta}} \right ).$$
Thus, taking into account the construction of the function $\psi
(\xi,t), $ we obtain
\begin{equation}
\label{31.6} w (\xi,t) \in H^{k+1+\ell,\frac{k+1+\ell}{2}}
 \left (\overline{Q^{\tilde{\varepsilon}, T}_{0,\frac{1}{2}}} \right ).
\end{equation}

Next, we fix $ \tilde{\delta} \in (0, 1/2)$  and choose, for
arbitrary $\delta_1 \in (\tilde{\delta}, 1/2),$  a function
$\tilde{\psi}_1 (\xi) \in C^\infty (\mathbb{R}) $ such that
$$ 0 \leq
\tilde{\psi}_1 (\xi) \leq 1, \quad  \tilde{\psi}_1 (\xi) =
\begin{cases} 1 &  \text{if} \;\; \xi \geq 1/2, \\   0  & \text{if}
\; \;  \xi \leq \delta_1.
\end{cases} $$
Now we set $\tilde{\psi} (\xi, t)  = \tilde{\psi}_1 (\xi)\, \psi_2
(t), $ where $\psi_2 (t)$ is given by (\ref{psi}). Then the function
$w_2 (\xi, t) = w (\xi, t) \, \tilde{\psi} (\xi, t)$ is a solution of
the boundary value problem
 \begin{equation}
\label{1.1a} \tilde{L} w_2 = F_2 (\xi,t), \quad (\xi, t) \in
Q^{\varepsilon, T}_{\tilde{\delta},1},
\end{equation}
 \begin{equation}
\label{1.2a} \partial_\xi w_2 \, (\tilde{\delta},t) = 0, \quad
\partial_\xi w_2 \, (1,t) = g(t) \psi_2 (t), \;\; \varepsilon \leq t
\leq T,
\end{equation}
 \begin{equation}
\label{1.3a} w_2 (\xi, \varepsilon) = 0, \quad \tilde{\delta} \leq
\xi \leq 1,
\end{equation}
where
\begin{equation}
\label{31.4a} F_2 (\xi, t) = F\, \tilde{\psi} + w \, \left (
\tilde{\psi}_t - a \tilde{\psi}_{\xi \xi} - b \tilde{\psi}_\xi \right
)  - 2a w_\xi \tilde{\psi}_\xi .
\end{equation}
From (\ref{21.0}) and (\ref{21.3}) it follows that  $F_2 (\xi, t) \in
H^{k-1+\ell,\frac{k-1+\ell}{2}}  \left (\overline{Q^{\varepsilon,
T}_{\tilde{\delta},1}} \right ). $ Therefore, by Theorem 5.3 of
\cite[Ch.4]{LUS}, Problem~(\ref{1.1a})--(\ref{1.3a}) has a unique
solution
$$w_2 (\xi,t) \in H^{k+1+\ell,\frac{k+1+\ell}{2}}
 \left (\overline{Q^{\varepsilon, T}_{\tilde{\delta},1}} \right ).$$
Now, taking into account the construction of the function
$\tilde{\psi} (\xi,t), $ we obtain
\begin{equation}
\label{31.6a} w (\xi,t) \in H^{k+1+\ell,\frac{k+1+\ell}{2}}
 \left (\overline{Q^{\tilde{\varepsilon}, T}_{\frac{1}{2}, 1}} \right ).
\end{equation}

Finally, (\ref{31.6}) and (\ref{31.6a}) imply that $w (\xi,t) \in
H^{k+1+\ell,\frac{k+1+\ell}{2}}
 \left (\overline{Q^{\tilde{\varepsilon}, T}_{0,1}} \right ),
$ which completes the proof of Lemma~\ref{lem1}.
\end{proof}
\bigskip

\begin{proof}[Proof of Proposition \ref{prop1}]
Since $v(\xi,t) \in C^{2,1} \left (\overline{Q_{0,1}^{0,T}} \right
),$ we have that $v, \, v_\xi,$ $\, v_{\xi \xi},$ $ v_t \in L^q \left
(\overline{Q_{0,1}^{0,T}} \right )$ for every $q>1.$ Therefore, by
Lemma~3.3 of \cite[Ch. 2]{LUS} it follows that $v_\xi (\xi,t)$ is
H\"older continuous in $t$ with exponent $ 1-3/q $ for every $ q>3.$
Hence (\ref{p.0}) holds.

We prove the assertion (b) by induction in $m.$ Let $m=2; $  suppose
that
$$
f(t) \in H^{2+\ell,\frac{2+\ell}{2}}
 \left (\overline{Q^{\varepsilon, T}_{0,1}} \right ), \quad \ell \in
 (0,1).
$$
From (\ref{p.0}) and (\ref{11.3}) it follows that
$$
s(t) \in H^{4+\ell,\frac{4+\ell}{2}}
 \left (\overline{Q^{\varepsilon, T}_{0,1}} \right ), \quad \ell \in
 (0,1).
$$
Now it is easy see that the coefficients of the operator $L$ in
(\ref{11.1}) satisfy the assumption (\ref{21.3}) in Lemma~\ref{lem1}
for $k=1.$ Therefore, applying Lemma~\ref{lem1} in the case $k=1$ to
the Problem (\ref{11.1})--(\ref{11.2}), we obtain that
$$
v(\xi,t) \in H^{2+\ell,\frac{2+\ell}{2}}
 \left (\overline{Q^{\tilde{\varepsilon}, T}_{0,1}} \right ) \quad
 \forall \tilde{\varepsilon} \in (\varepsilon, T).
$$
Then, in view of (\ref{11.3}), we conclude that
$$
s(t) \in H^{5+\ell,\frac{5+\ell}{2}}
 \left (\overline{Q^{\tilde{\varepsilon}, T}_{0,1}} \right ) \quad
 \forall \tilde{\varepsilon} \in (\varepsilon, T).
$$
Hence, (\ref{p.1}) and (\ref{p.2}) hold for $m=2, $  i.e., the
assertion (b) holds for $m=2.$

Assume that (b) holds for some $m\geq 2;$ we shall prove that (b)
holds for $m+1.$  Let $f(t) \in H^{m+1+\ell,\frac{m+1+\ell}{2}}
 \left (\overline{Q^{\varepsilon, T}_{0,1}} \right ).$
Then from the inductive hypothesis it follows that (\ref{p.1}) and
(\ref{p.2}) hold. Therefore, the coefficients of the operator $L$ in
(\ref{11.1}) satisfy (\ref{21.3}) with $k=m.$ Thus, by
Lemma~\ref{lem1} we conclude that $v(\xi,t) \in
H^{m+1+\ell,\frac{m+1+\ell}{2}}
 \left (\overline{Q^{\tilde{\varepsilon}, T}_{0,1}} \right ) \;
 \forall \tilde{\varepsilon} \in (\varepsilon, T),$
i.e., (\ref{p.1}) holds for $m+1.$ Now, in view of (\ref{11.3}), we
obtain that $ s(t) \in H^{m+4+\ell,\frac{m+4+\ell}{2}}
 \left (\overline{Q^{\tilde{\varepsilon}, T}_{0,1}} \right ), $
$\forall \tilde{\varepsilon} \in (\varepsilon, T).$ Hence,
(\ref{p.2}) holds for $m+1$ as well. This completes the proof of
Proposition~\ref{prop1}.
\end{proof}

\bigskip

\begin{proof}[Proof of Theorem \ref{thm1}] If $f(t) \in C^m ((0,T]) $ for
some $m>1,$  then $$f(t) \in H^{2m-1+\ell,\frac{2m-1+\ell}{2}}
 \left (\overline{Q^{\varepsilon, T}_{0,1}} \right ) \quad \forall
 \varepsilon \in (0,T),  \;  \forall \ell \in (0,1).$$
Now, for every fixed $\varepsilon >0, $ Proposition \ref{prop1}
implies that
$$s(t) \in
H^{2m+2+\ell,\frac{2m+2+\ell}{2}}
 \left (\overline{Q^{\tilde{\varepsilon}, T}_{0,1}} \right ) \quad
  \forall \tilde{\varepsilon} \in (\varepsilon, T).$$
Thus, it follows that $ s(t) \in C^{m+1} ((0,T]).$ This completes the
proof of Theorem~\ref{thm1}.
\end{proof}

\section{Upper bounds for the smoothness of $s(t)$ }

Now we are going to explain that the smoothness of $s(t) $ (in terms
of H\"older scale) is bounded above by the smoothness of $f(t).$

\begin{Proposition}
\label{prop2} Let $v(\xi, t) $ be the function defined by
(\ref{11.0}) (and satisfying (\ref{11.1})--(\ref{11.3})).
 Then for every $m \in \mathbb{N},$ $\varepsilon >0$ and $T> \varepsilon $
 the following implication holds:
\begin{equation}
\label{s1} s(t) \in H^{m+4+\ell,\frac{m+4+\ell}{2}}
 \left (\overline{Q^{\varepsilon, T}_{0,1}} \right )  \Longrightarrow
v(\xi, t) \in H^{m+1+\ell,\frac{m+1+\ell}{2}}
 \left (\overline{Q^{\tilde{\varepsilon}, T}_{0,1}} \right ) \;\;
 \forall \tilde{\varepsilon} \in (\varepsilon, T).
\end{equation}
\end{Proposition}

In the proof of Proposition~\ref{prop2} we need the following
statement.
\begin{Lemma}
\label{lem2} Let $\varepsilon>0, $  $\ell \in (0,1), \; k \in
\mathbb{N}, $
\begin{equation}
\label{21.0a} w(\xi, t) \in H^{k+\ell,\frac{k+\ell}{2}}
 \left (\overline{Q^{\varepsilon, T}_{0,1}} \right )
 \end{equation}
and $w_{\xi \xi} $ exists in $ Q^{\varepsilon, T}_{0,1}$ in the case
$k=1,$ and let $w(\xi,t)$ satisfy the equation
 \begin{equation}
\label{21.1a} \tilde{L} w:=w_t - a(\xi,t) w_{\xi \xi} - b (\xi, t)
w_\xi - c (\xi, t) w = F(\xi,t), \quad  (\xi, t) \in Q^{\varepsilon,
T}_{0,1},
\end{equation}
where
 \begin{equation}
\label{21.3a} a, b, c, F  \in H^{k-1+\ell,\frac{k-1+\ell}{2}}
 \left (\overline{Q^{\varepsilon, T}_{0,1}} \right ), \quad a(\xi,t) \geq const
 >0.
  \end{equation}
If
 \begin{equation}
\label{21.2a} w_\xi (0,t)= h(t), \quad w_\xi (1,t) = g(t), \quad
\varepsilon \leq t \leq T,
\end{equation}
with
 \begin{equation}
\label{21.4a} h(t) \in H^{k+\ell,\frac{k+\ell}{2}}
 \left (\overline{Q^{\varepsilon, T}_{0,1}} \right ), \quad
 g(t)  \in H^{k+\ell,\frac{k+\ell}{2}}
 \left (\overline{Q^{\varepsilon, T}_{0,1}} \right ),
 \end{equation}
then
\begin{equation}
\label{21.5a} w(\xi, t) \in H^{k+1+\ell,\frac{k+1+\ell}{2}}
 \left (\overline{Q^{\tilde{\varepsilon}, T}_{0,1}} \right ) \quad \forall
 \tilde{\varepsilon} \in (\varepsilon, T).
 \end{equation}
\end{Lemma}

\begin{proof}
The proof of this statement  is similar to the proof of
Lemma~\ref{lem1}. Indeed, let $\tilde {\varepsilon} \in (\varepsilon,
T); $ choose $\varepsilon_1 \in (\varepsilon, \tilde{\varepsilon}) $
and $\hat{\psi} (t) \in C^\infty (\mathbb{R}) $ such that $
\hat{\psi} (t) = 1 $ for $t \geq \tilde{\varepsilon}$ and $\hat{\psi}
(t) =0 $ for $t \leq \varepsilon_1.$ Set $\tilde{w}(\xi, t) = w(\xi,
t) \cdot \hat{\psi} (t); $ then the function $\tilde{w}(\xi, t)$ is a
solution of the boundary value problem
$$
\tilde{L} \tilde{w} = \tilde{F}, \quad  \tilde{F} (\xi, t) = F(\xi,
t) \cdot \hat{\psi} (t) + w (\xi, t) \hat{\psi}^\prime (t), \quad
(\xi, t) \in Q^{\varepsilon, T}_{0,1},
$$
$$
\tilde{w}_\xi (0,t) = h(t)\hat{\psi}(t), \quad \tilde{w}_\xi (1,t) =
g(t)\hat{\psi}(t), \quad \tilde{w} (\xi, 0) =0.
$$
From (\ref{21.0a}) and (\ref{21.3a}) it follows that  $\tilde{F}
(\xi, t) \in H^{k-1+\ell,\frac{k-1+\ell}{2}}  \left
(\overline{Q^{\varepsilon, T}_{0,1}} \right ). $ Now, by Theorem 5.3
of \cite[Ch.4]{LUS}, we conclude that the above boundary value
problem has a unique solution
$$ \tilde{w} (\xi,t) \in H^{k+1+\ell,\frac{k+1+\ell}{2}}
 \left (\overline{Q^{\varepsilon, T}_{0,1}} \right ).$$
Thus, taking into account that $\hat{\psi}(t)=1 $ for $t \geq
\tilde{\varepsilon},$ we obtain that $w(\xi, t) \in
H^{k+1+\ell,\frac{k+1+\ell}{2}}
 \left (\overline{Q^{\tilde{\varepsilon}, T}_{0,1}} \right ).$
The proof of Lemma~\ref{lem2} is complete.
\end{proof}

\begin{proof}[Proof of Proposition \ref{prop2}] We prove the claim by
induction in $m.$

Let $m=1; $ then we assume that $s(t) \in H^{5+\ell,\frac{5+\ell}{2}}
 \left (\overline{Q^{\varepsilon, T}_{0,1}} \right ) $
and prove that $v(\xi, t) \in H^{2+\ell,\frac{2+\ell}{2}}
 \left (\overline{Q^{\tilde{\varepsilon}, T}_{0,1}} \right ) \;\;
 \forall \tilde{\varepsilon} \in (\varepsilon, T).$
Indeed, in view of (\ref{11.2}) and (\ref{11.3}), the function
$v(\xi,t) $ satisfies the boundary conditions
\begin{equation}
\label{bc} v_{\xi} (0,t)=h(t), \quad v_{\xi} (1,t)=0, \quad
\varepsilon \leq t \leq T,
\end{equation}
where
\begin{equation}
\label{h} h(t):= s(t) \, \left [(v(1,t)-\lambda -\sigma) s^\prime
(t)-\lambda \sigma s(t) - s^{\prime \prime}(t)\right ].
\end{equation}

From the above assumptions on $s(t),$ and from the assertion (a) in
Proposition~\ref{prop1}, it follows that $$h(t) \in
H^{1+\ell,\frac{1+\ell}{2}}
 \left (\overline{Q^{\varepsilon, T}_{0,1}} \right ).$$
Moreover, since $s^\prime (t) \in H^{3+\ell,\frac{3+\ell}{2}}
 \left (\overline{Q^{\varepsilon, T}_{0,1}} \right ),$
we obtain in view of (\ref{11.1}) that the coefficients of the
operator $L$ belong to the space $H^{3+\ell,\frac{3+\ell}{2}}
 \left (\overline{Q^{\varepsilon, T}_{0,1}} \right ).$

Therefore,  applying Lemma \ref{lem2} in the case when $k=1,$
$\tilde{L}=L,$ $w=v$ and boundary conditions given by (\ref{bc}), we
conclude that
$$v(\xi, t) \in H^{2+\ell,\frac{2+\ell}{2}}
 \left (\overline{Q^{\tilde{\varepsilon}, T}_{0,1}} \right ) \quad
 \forall \tilde{\varepsilon} \in (\varepsilon, T),$$
i.e., (\ref{s1}) holds for $m=1.$

Assume that the assertion holds for some $m \geq 1. $ We will prove
that (\ref{s1}) holds for $m+1.$ Suppose
$$
s(t) \in H^{m+5+\ell,\frac{m+5+\ell}{2}}
 \left (\overline{Q^{\varepsilon, T}_{0,1}} \right ).
$$
Then from the inductive hypothesis it follows that
$$
v(\xi,t) \in H^{m+1+\ell,\frac{m+1+\ell}{2}}
 \left (\overline{Q^{\tilde{\varepsilon}, T}_{0,1}} \right )  \quad
 \forall \tilde{\varepsilon} \in (\varepsilon,T).
$$
Therefore, in view of (\ref{h}) one can easily see that
$$
h(t) \in H^{m+1+\ell,\frac{m+1+\ell}{2}}
 \left (\overline{Q^{\tilde{\varepsilon}, T}_{0,1}} \right )  \quad
 \forall \tilde{\varepsilon} \in (\varepsilon,T).
$$
Hence, applying Lemma~\ref{lem2} in the case $k=m+1 $ we conclude
that
$$
v(\xi,t) \in H^{m+2+\ell,\frac{m+2+\ell}{2}}
 \left (\overline{Q^{\tilde{\varepsilon}, T}_{0,1}} \right )  \quad
 \forall \tilde{\varepsilon} \in (\varepsilon,T),
$$
i.e., (\ref{s1}) holds for $m+1.$ This completes the proof of
Proposition~\ref{prop2}.
\end{proof}

\begin{Theorem}
\label{thm2}
 Suppose the pair of functions $(u(x,t),
s(t))$ is a classical solution of Problem $P$ satisfying (i)-(iii).
 If $s(t) \in C^{m+2} ((0,T]),\; m \in \mathbb{N}, \;T >0, $
 then $f(t) \in C^m ((0,T]).$

Moreover, if $s(t) \in C^\infty ((0,\infty)),$ then $f(t) \in
C^\infty ((0,\infty)).$
\end{Theorem}

\begin{proof}
Suppose $s(t) \in C^{m+2} ((0,T]);$ then  $$s(t) \in
H^{2m+3+\ell,\frac{2m+3+\ell}{2}}
 \left (\overline{Q^{\varepsilon, T}_{0,1}} \right )  \quad
 \forall \varepsilon \in (0,T).$$
Therefore, by Proposition~\ref{prop2} we obtain that
$$
v(\xi,t) \in H^{2m+\ell,\frac{2m+\ell}{2}}
 \left (\overline{Q^{\tilde{\varepsilon}, T}_{0,1}} \right )  \quad
 \forall \tilde{\varepsilon} \in (\varepsilon,T).
$$
Since $f(t) = v(0,t), $ it follows that
$$
f(t) \in H^{2m+\ell,\frac{2m+\ell}{2}}
 \left (\overline{Q^{\tilde{\varepsilon}, T}_{0,1}} \right )  \quad
 \forall \tilde{\varepsilon} \in (\varepsilon,T);
$$
thus $f(t) \in C^m ([\tilde{\varepsilon},T]) \quad
 \forall \tilde{\varepsilon} \in (\varepsilon,T), \; 0< \varepsilon <T, $
which implies that $f(t) \in C^m ((0,T]).$ The proof of
Theorem~\ref{thm2} is complete.
\end{proof}

Theorem~\ref{thm1} and Theorem~\ref{thm2} prove, respectively, that
the condition $f(t) \in C^\infty ((0,\infty))$ is necessary and
sufficient for $s(t) \in C^\infty ((0,\infty)).$ In other words, the
following holds.

\begin{Corollary}
\label{cor1} In Problem P, the free boundary $x=s(t), \; t \in
(0,\infty)$  is an infinitely differentiable curve if and only if
$f(t) \in C^\infty ((0,\infty)).$
\end{Corollary}

\end{document}